\documentclass[12pt, leqno]{amsart}
\usepackage[mathcal]{euscript}
\usepackage{amssymb, amsfonts, xy}
\usepackage{graphicx}

\usepackage{setspace}

\xyoption{all} \CompileMatrices

\begin{document}

\def\map{\mathop{\rm map}\nolimits}
\def\ker{\mathop{\rm ker}\nolimits}
\def\im{\mathop{\rm im}}
\def\precdot{\mathop{\prec\!\!\!\cdot}\nolimits}
\def\colim{\mathop{\rm colim}}
\def\lim{\mathop{\rm lim}}
\def\ohom{\mathop{\rm \overline{Hom}}\nolimits}
\def\Hom{\mathop{\rm Hom}\nolimits}
\def\Ext{\mathop{\rm Ext}\nolimits}
\def\Rep{\mathop{\rm Rep}\nolimits}
\def\Aut{\mathop{\rm Aut}}
\def\id{\mathop{\rm id}\nolimits}
\def\Ast{\mathop{*}}
\def\vert{\mathop{\textsf{vert}}\nolimits}
\def\edge{\mathop{\textsf{edge}}\nolimits}
\def\aa{\mathbb{A}}
\def\g{\mathop{\mathcal Graphs}\nolimits}
\def\a{\mathop{\mathcal Ab}\nolimits}

\renewcommand{\theenumi}{\alph{enumi}}

\def\presup#1{{}^{#1}\hspace{-.11em}}      %pre-superscript
\def\presub#1{{}_{#1}\hspace{-.08em}}      %pre-subscript

\def\intertitle#1{

\medskip

{\em \noindent #1}

\smallskip
}

\newtheorem{thm}{Theorem}[section]
\newtheorem{theorem}[thm]{Theorem}
\newtheorem{lemma}[thm]{Lemma}
\newtheorem{corollary}[thm]{Corollary}
\newtheorem{proposition}[thm]{Proposition}
\newtheorem{example}[thm]{Example}

\theoremstyle{remark}
\newtheorem{rem}[thm]{Remark}

\theoremstyle{definition}
\newtheorem{definition}[thm]{Definition}

\makeatletter
\let\c@equation\c@thm
\makeatother
\numberwithin{equation}{section}

%%%%%%%%%%%%%%%%%%%%%%%%%%%%%%%%%%%%%%%%%%%%%%%%%%%%%%%%%%%%%%%%%%%%%%%%%%
% temporary commands:

%\doublespace

\newcommand{\comment}[1]{}

\newif\ifshortlabels\shortlabelstrue

\ifshortlabels
  \newcommand{\mylabel}[1]{\label{#1}}
  \newcommand{\mylabeln}[1]{\label{#1}}
\else
  \newcommand{\mylabel}[1]{(#1) \label{#1}\ \ \ \ \ \ \ \ \ \ \ \ \ \ \\}
  \newcommand{\mylabeln}[1]{(#1) \label{#1}}
\fi

%%%%%%%%%%%%%%%%%%%%%%%%%%%%%%%%%%%%%%%%%%%%%%%%%%%%%%%%%%%%%%%%%%%%%%%%%%

%\begin{frontmatter}
% Title, authors and addresses

\title{An almost full embedding \\ of the category of graphs \\
  into the category of abelian groups}
\author{Adam J. Prze\'zdziecki$^1$}
\address{Department of Mathematics, Warsaw University of Life Sciences - SGGW, Warsaw, Poland}
\email{adamp@mimuw.edu.pl}

\maketitle
\begin{center}
\today
\end{center}

\comment{
\begin{center}
\vspace{-0.7cm}
{\small \it Institute of Mathematics, Warsaw University, Warsaw, Poland \\
Email address: \verb+adamp@mimuw.edu.pl+ } \\
\end{center}
}

\footnotetext[1]{The author was partially supported by grant
  N N201 387034 of the Polish Ministry of Science and Higher Education.}

\begin{abstract}
  We construct an embedding $G:\g\to{\mathcal A}b$
  of the category of graphs into the category of abelian groups such that for $X$ and $Y$ in $\g$ we have
  $$\Hom(GX,GY)\cong\mathbb{Z}[\Hom_{\g}(X,Y)],$$
  the free abelian group whose basis is the set $\Hom_{\g}(X,Y)$. The isomorphism is functorial in $X$ and $Y$. The existence of such an embedding implies that, contrary to a common belief, the category of abelian groups is as complex and comprehensive as any other concrete category. We use this embedding to settle an old problem of Isbell whether every full subcategory of the category of abelian groups, which is closed under limits, is reflective. A positive answer turns out to be equivalent to weak Vop\v enka's principle, a large cardinal axiom which is not provable but believed to be consistent with standard set theory. We obtain some consequences to the Hovey-Palmieri-Stricland problem about existence of arbitrary localizations in a stable homotopy category.

  Several known constructions in the category of abelian groups are obtained as quick applications of the embedding.

\vspace{4pt}
  {\noindent\leavevmode\hbox {\it MSC:\ }}
  {18B15 (20K20,20K30,18A40,03E55)} \\
  \comment{
  20K20 Torsion-free groups, infinite rank \\
  20K30 Automorphisms, homomorphisms, endomorphisms, etc. \\
  03E55 Large cardinals \\
  18A40 Adjoint functors (universal constructions, reflective subcategories, Kan extensions, etc. \\
  18B15 Embedding theorems, universal categories. \\
  }
  \vspace{-9pt}

\end{abstract}

%\end{frontmatter}

\section{Introduction}
\mylabel{section-introduction}

A classical theorem of Corner \cite[Theorem A]{corner} states that for every countable ring $A$ there exists an abelian group $\mathbb{A}$ such that $A\cong\Hom(\mathbb{A},\mathbb{A})$. Actually, the proof of this theorem implies that analogous groups $\mathbb{A}$ exist also for rings $A$ of cardinality continuum whose additive group is free. We start with such a group which corresponds to a ring whose additive group is generated by morphisms between at most countable graphs. By means of colimits we construct out of this group a functor
$$G:\g\to\a$$
which embeds the category of graphs into the category of abelian groups in such a way that for $X$ and $Y$ in $\g$ we have a natural equivalence
$$\gamma:\mathbb{Z}[\Hom_{\g}(X,Y)]\stackrel{\cong}{\longrightarrow}\Hom(GX,GY)$$
where $\mathbb{Z}[S]$ denotes the free abelian group with basis $S$. The $\mathbb{Z}[S]$ notation is chosen so as to resemble the group ring, which we have when $\Hom_{\g}(X,X)$ is a group.

The target category $\a$ is often regarded as being ``easy'', a category in which analogues of hard problems, formulated in other categories, should have simple solutions. In fact, the problem which motivates this paper was summed up with a little surprise as being ``an unsolved problem even in the category of abelian groups'' \cite[page 620]{isbell-categories}. On the contrary, the category of graphs is actively researched since it is nearly universal for concretizable categories. A category is concretizable if it admits an embedding, not necessarily full, into the category of sets. Concrete categories are often called ``point-set'' categories. Corollary 4.5 in \cite[Chapter III]{trnkova-book} implies that if we assume that measurable cardinals are bounded above, then every concretizable category fully embeds in $\g$. The assumption about an upper bound for measurable cardinals is consistent with standard ZFC set theory. Actually, most of the concrete categories encountered in the literature are accessible ---  such categories embed into $\g$ without any nonstandard set-theoretic axioms \cite[Theorem 2.65]{adamek-rosicky}.

We see that the existence of the functor $G$ implies that the category of abelian groups is as complex and comprehensive as any other concrete category. Any construction possible in the category of graphs, or another concrete category, whose properties are not destructed when addition of morphisms is introduced, translates into abelian groups. Several quick examples of such reasonings are listed in Section \ref{section-applications}. These are mostly alternative proofs of known theorems.

Our main motivation for this paper was to answer an old problem of Isbell \cite[pages 619-620]{isbell-categories} whether every embedding of a full subcategory, closed under taking limits, into the category of abelian groups has a left adjoint. The answer turns out to be equivalent to weak Vop\v enka's principle. A negative answer is consistent with ZFC while a positive one is believed to be consistent with but not provable in ZFC. Section \ref{section-orthogonal} is concerned with this problem, known as the orthogonal subcategory problem.

In Section \ref{section-stable} we derive some implications of our construction to the Hovey-Palmieri-Stricland question \cite{hovey-axiomatic} about existence of arbitrary localizations in a stable homotopy category.

This paper adds to an already significant literature on embeddings of categories. Let us recall here a choice of results. A full embedding of $\g$ into the category of semigroups was constructed in \cite{semigroups} and into integral domains in \cite{integral}. Almost full embeddings, up to constant maps, into the category of metric spaces were obtained in \cite{trnkova-metric} and into paracompact spaces in \cite{paracompact}. Almost full embeddings: up to nullhomotopic maps into the the homotopy category and up to trivial homomorphisms and conjugation into the category of groups were constructed in \cite{przezdziecki-groups}. For other related results see \cite{trnkova-book}, \cite{adamek-rosicky}, \cite{trnkova-30}.

The embedding $G$ is constructed in Section \ref{section-embedding}. It has nice categorical properties, but the fact that it takes values in groups of cardinality at least continuum makes it not very flexible, especially when we need an inductive argument or the least example possible. To mitigate this problem we introduce, in Section \ref{section-approximation}, another embedding $G_{fin}$ which approximates $G$. We have a natural transformation $G_{fin}\to G$ consisting of monomorphisms. For infinite graphs $X$ we have $|X|=|G_{fin}X|$. The price we pay is that the natural transformation
$$\gamma:\mathbb{Z}[\Hom_{\g}(X,Y)]\longrightarrow\Hom(G_{fin}X,G_{fin}Y)$$
is one-to-one but not onto in general. However it does happen that the homomorphism $\gamma$ is an isomorphism, for example when $X$ or $\Hom(X,Y)$ is finite.

Unless explicitly stated, every group in this paper is torsion free abelian. Countable means at most countable. We write $\Hom(X,Y)$ rather than $\Hom_{\mathcal{C}}(X,Y)$ when there is no need to emphasize the category to which the morphisms belong. A free group on a basis $S$ is denoted $\mathbb{Z}[S]$.

\section{Corner's method}

In this section we formulate a theorem implicitly proved by Corner
\cite[Theorem A]{corner} and crucial for the following section.

Let $\widehat{A}=\lim_nA/nA$ denote the {\em natural completion} of the group $A$. The limit is taken over positive integers $n$ ordered by divisibility. A classical reference for natural completion, called also a {\em $Z$-adic completion}, is \cite[Theorem 39.5]{fuchs}. We have a group isomorphism $\widehat{A}\cong\prod A^{\wedge}_p$. If $A$ is the additive group of a ring, and $A$ is {\em reduced} (i.e. $\Hom(\mathbb{Q},A)=0)$, then $A\subseteq\widehat{A}$ and $\widehat{A}$ admits a unique ring structure which extends $A$. A subgroup $C$ of $B$ is called {\em pure} if for every integer $n$ and $b\in B$ the condition $nb\in C$ implies $b\in C$. Corner's theorem rests on the following beautiful observation (see \cite[Lemma 1.2]{corner}).

\begin{lemma} \mylabel{lemma-pure-subgroup}
  If $C$ is a pure subgroup of $\widehat{A}$ which contains $A$, then $\widehat{C}=\widehat{A}$. In particular, the inclusion $A\subseteq C$ induces an isomorphism $\Hom(A,\widehat{A})\cong\Hom(C,\widehat{A})$.
\end{lemma}

Lemma 1.5 in \cite{corner} implies the following.

\begin{lemma}
  The natural completion of the integers, the ring $\widehat{\mathbb{Z}}$, has a pure subring $\bf P$ without zero divisors such that $|{\bf P}|=2^{\aleph_0}$.
\end{lemma}

The following is essentially proved in \cite[Theorem A]{corner}.

\begin{theorem} \mylabel{theorem-corner}
  If $A$ is a ring of cardinality at most continuum whose additive group is free then there exists a group ${\aa}$ such that:
  \begin{enumerate}
    \item \label{t-corner-inclusions} $A\subseteq\aa\subseteq\widehat{A}$ as left $A$-modules.
    \item \label{t-corner-ring} $A\cong\Hom({\aa},{\aa})$.
    \item \label{t-corner-cardinality} $|\mathbb{A}|=|A|$.
  \end{enumerate}
\end{theorem}

We need two comments on the original proof of \cite[Theorem A]{corner}. Corner's construction requires that the transcendence degree of $\bf P$ over a certain subring $\bf \Pi$ of $\bf P$, associated with the group $A$, is at least $|A|$. If the additive group of $A$ is free then $\bf \Pi$ is the ring of integers $\mathbb{Z}$ and this condition is satisfied. Secondly, a trivial difference: Corner considers the right action of $A$ on $\aa$; here, since functions act on graphs on the left, we prefer the left action.

This theorem is used throughout most of the paper as a ``black box''. One place where we look again into Corner's construction is the comment below Remark \ref{remark-presentation-gx}, which is not used in other arguments. Another place is Section \ref{section-approximation} where we introduce certain pure subring $A_{fin}$ of $A$ -- we need to know that if $A_0\subseteq A$ is a pure subring then the construction of $\aa$ for $A$ restricts to $\aa_0$ for $A_0$, and $\aa_0\subseteq\aa$ is an inclusion of $A_0$-modules.

\section{Embedding of the category of graphs}
\mylabel{section-embedding}

A graph $X$ is a set, denoted with the same letter $X$, with a binary relation $R\subseteq X\times X$.
We construct a functor $G$ from the category of graphs to the category of abelian groups such that for graphs $X$ and $Y$ we have $\Hom_{\a}(GX,GY)\cong\mathbb{Z}[\Hom_{\g}(X,Y)]$, the free group whose basis is $\Hom_{\g}(X,Y)$.

A poset $I$ is called {\em directed (resp. countably directed)} if any finite
subset (resp. any countable subset) of $I$ has an upper bound in $I$. A poset is viewed as a category where $a\leq b$ corresponds to a morphism $a\to b$.
A diagram (i.e. functor) $S:I\to\mathcal{C}$ and its colimit $\colim S$
are called {\em (countably) directed} if $I$ is (countably) directed. A diagram $S$ and its limit $\lim S$ are called {\em (countably) codirected} if the opposite category $I^{op}$ is (countably) directed.

Let $\Gamma$ be a full subcategory of $\g$ whose objects are representatives of the isomorphism classes of countable graphs. Clearly $\Gamma$ has the cardinality of the continuum.
Let $A=\mathbb{Z}[\Gamma]$ be the ring whose additive group is free with the basis consisting of the identity $1$ and maps $\varphi:X\to Y$ in $\Gamma$. If $\psi:X'\to Y'$ is another member of the basis then the product $\varphi\psi$ in $A$ is the composition $\varphi\psi$ if $Y'=X$, that is $\varphi$ and $\psi$ are composable, and otherwise zero. Let ${\aa}$ be the group described in Theorem \ref{theorem-corner}. This construction gives us inclusions: $(\mbox{morphisms of }\Gamma)\subseteq A\subseteq \mathbb{A}\subseteq\hat{A}$.

{\em Convention.} We use the same letters to denote the maps in $\Gamma$, corresponding elements of $A$ and endomorphisms of $\aa$. When it is clear from context and simplifies notation, we write $\sigma$ instead of $G\sigma$ or $G_{fin}\sigma$.

Let $\id_X:X\to X$ be the identity map. If $X$ is in $\Gamma$ then $\id_X$ is an idempotent of $A$ and we have $\aa\cong \id_X\aa\oplus(1-\id_X)\aa$.

\begin{definition} \mylabel{definition-g}
\begin{itemize}
  \item[(a)]
  For $X$ in $\Gamma$ we define $GX=\id_X\aa$. If $\varphi:X\to Y$ is a map in $\Gamma$ then $\id_Y\varphi=\varphi$, hence
  $\varphi \id_X\aa\subseteq \id_Y\aa$, and therefore $\varphi$ induces, via left multiplication, a group homomorphism $GX\to GY$. Thus $G$ is a functor from $\Gamma$ to $\a$.
  \item[(b)]
  If $X$ is an arbitrary graph then we define
  $$
    GX=\colim_{c\in\Gamma\downarrow X}Gc
  $$
  where $\Gamma\downarrow X$ is the category of maps $c:C\to X$ with $C$ in $\Gamma$ and $Gc$ is defined as $GC$.
  We may view this as an extension of (a) since for $X$ in $\Gamma$ the category $\Gamma\downarrow X$ contains a terminal object $\id_X$.
\end{itemize}
\end{definition}

\rem \mylabel{remark-colimit} For an arbitrary graph $X$ we have
  $$%\begin{equation}\mylabel{equation-colimit}
    GX=\colim_{C\in[X]}GC.
  $$
  Where $[X]$ denotes the category of inclusions whose objects are countable subgraphs $C\subseteq X$.
  This is clear sinnce $[X]$ is isomorphic to a cofinal subcategory of $\Gamma\downarrow X$.
  A map $X\to Y$ induces, by taking images, a map $[X]\to[Y]$, which in turn induces a map $Gf$ of colimits.

\rem The only reason we need to include infinite subgraphs in $[X]$ is that we want this poset to be countably directed to apply Lemma \ref{lemma-limits-zs} in the proof of Theorem \ref{theorem-embedding-graphs}. All statements of this section except \ref{lemma-limits-zs} and \ref{theorem-embedding-graphs} hold when we restrict $\Gamma$ and $[X]$ to finite graphs. In the next section we discuss the properties of the functor $G_{fin}$ obtained by means of finite graphs.

We need a closer look at the structure of the groups $GX$ for $X$ in $\Gamma$.

\rem \mylabel{remark-gamma-x}
  Let $\Gamma_X$ denote the set of maps $\sigma:C\to X$ in $\Gamma$. It is a subset of the morphisms of $\Gamma$ hence a subset of the basis of $A$ viewed as a group.
  Left multiplication by $\id_X$ is the identity on $\Gamma_X$ and zero on $1-\id_X$ and on $\Gamma\setminus\Gamma_X$.
  Thus $\id_X$ induces a projection of $A$ onto the (free) subgroup $\langle\Gamma_X\rangle$, generated by $\Gamma_X$.
  Therefore, applying Theorem \ref{theorem-corner}, we have:
  $$
  \langle\Gamma_X\rangle\subseteq GX\subseteq\widehat{\langle\Gamma_X\rangle}
  $$

A morphism $\sigma:C\to X$ in $\Gamma$ may be viewed as an element of the ring $A\subseteq\aa$ which induces a homomorphism $\sigma:GC\to GX$, or as an element of $\Gamma_X\subseteq GX$. In Lemma \ref{lemma-dense} we are going to see that $\id_X\in GX$ is particularly important.

\rem \mylabel{remark-presentation-gx}
  Every element $u\in GX$ may be uniquely written as $u=\sum z_i\sigma_i$ where $\sigma_i\in\Gamma_X$ and $z_i\in\widehat{\mathbb{Z}}$. The membership $u\in\widehat{\langle\Gamma_X\rangle}$ implies that the elements $z_i$ form a set which is finite or countable with $0$ as its unique accumulation point in $\widehat{\mathbb{Z}}$.

Looking at the construction of $\aa$ in \cite[proof of Theorem A, after (3)]{corner}, we see that the set of $z_i$'s above is finite: The group $\aa$ is defined there as the pure subgroup of $\widehat{A}$ generated by $A$ and subsets of the form $Ae_a$ where $a\in A$ and $e_a=z_a\cdot 1+w_a\cdot a$ for certain $z_a$ and $w_a$ in $\widehat{\mathbb{Z}}$. This implies that the number of $z_i\neq 0$ has to be finite.

\begin{lemma}\mylabel{lemma-injective}
  If $\varphi:X\to Y$ is one-to-one then so is $G\varphi:GX\to GY$.
\end{lemma}
\begin{proof}
  First we prove the case when $\varphi$ is in $\Gamma$. Remark \ref{remark-gamma-x} implies that
  $\langle\Gamma_X\rangle\subseteq GX\subseteq\widehat{\langle\Gamma_X\rangle}$.
  Injectivity of $\varphi:X\to Y$ implies injectivity of the induced map $\Gamma_X\to\Gamma_Y$, then injectivity of the homomorphism $\widehat{\langle\Gamma_X\rangle}\to\widehat{\langle\Gamma_Y\rangle}$, and then injectivity of its restriction $G\varphi:GX\to GY$.

  Since $[X]$ is a category of inclusions, the previous paragraph implies that (\ref{remark-colimit}) is a directed colimit of inclusions and therefore $GC\subseteq GX$ for every countable $C\subseteq X$. For any $x\neq y$ in $GX$, there exists a countable $C\subseteq X$ such than $x$ and $y$ belong to $GC$. Then $GC$ is mapped isomorphically to $G(\varphi C)$ and, since analogously $G(\varphi C)\subseteq GY$, we obtain $G\varphi(x)\neq G\varphi(y)$.
\end{proof}

Note that $G$ hardly ever preserves epimorphisms. This is so because maps to $\varphi(X)$ in $\Gamma$ usually do not lift to $X$.

\begin{lemma}\mylabel{lemma-left-multiplication}
  Every homomorphism $h:GX\to GY$ with $X$ and $Y$ in $\Gamma$ can be uniquely represented as left multiplication by an $a\in A$. We have $a=\sum_{i\in I}k_i\sigma_i$, where $k_i$ are nonzero integers, $\sigma_i:X\to Y$ are distinct maps in $\Gamma$, and $I$ is finite.
\end{lemma}

\begin{proof}
  Let $r:\aa\to\aa$ be the composition $\aa\stackrel{\id_X\cdot(-)}{-\!\!\!\longrightarrow}GX\stackrel{h}{\longrightarrow}GY\subseteq\aa$.
  Theorem \ref{theorem-corner}(\ref{t-corner-ring}) implies that $r(x)=ax$ for some $a$ in $A$.
  By the definition of $A$ we have a unique representation $a=\sum_{i\in I}k_i\sigma_i$ where $k_i$ are integers, $\sigma_i:X_i\to Y_i$ are distinct maps in $\Gamma$ (or the identity), and $I$ is finite.
  Since $r=\id_Yr\id_X$ we see that $X_i=X$ and $Y_i=Y$ for $i\in I$, and the $\sigma_i$ above can not be equal to the identity of $A$.
\end{proof}

The following two lemmas are tautological thanks to the inclusion $A\subseteq\aa$ as $A$-modules.

\begin{lemma} \mylabel{lemma-dense}
  For any homomorphism $h:GX\to GY$ with $X$ and $Y$ in $\Gamma$, if $h(\id_X)=\{0\}$ then $h=0$.
\end{lemma}
\begin{proof}
   Lemma \ref{lemma-left-multiplication} yields an
   $a=\sum_{i\in I}k_i\sigma_i$ such that $h(x)=ax$ for $x\in GX$, and $\sigma_i:X\to Y$ are distinct maps in $\Gamma$.
   Then $h(\id_X)=\sum k_i\sigma_i\id_X=\sum k_i\sigma_i\in GY$.
   The assumption that $h(\id_X)=0$ implies $k_i=0$ for all $i$ and therefore $a=0$ and $h=0$.
\end{proof}

\begin{lemma}\mylabel{lemma-factors}
  If a graph $W$ in $\Gamma$, a monomorphism $\varphi:X\to Y$ in $\Gamma$ and a homomorphism $h:GX\to GY$ fit into the following diagram
  $$
  \xymatrix{
    {\{\id_W\}} \ar[r]
        \ar@{}[d]|(0.44){\rotatebox{-90}{$\subseteq$}} &
      GX \ar[d]^{G\varphi} \\
    GW \ar[r]^h \ar@{-->}[ur]^{\tilde{h}} &
      GY
  }
  $$
  then the dashed arrow $\tilde{h}$ exists.
\end{lemma}

\begin{proof}
  As in Lemma \ref{lemma-dense}, we have $h(\id_W)=\sum k_i\sigma_i\in GY$. There exists $u\in GX$ such that $h(\id_W)=G\varphi(u)$. Remark \ref{remark-presentation-gx} implies that $u$ is uniquely represented as $\sum z_j\tau_j$. We obtain
  $$\sum k_i\sigma_i=\sum z_j\varphi\tau_j$$
  hence each $\sigma_i$ is of the form $\varphi\tau_j$. Since $\varphi$ is a monomorphism these factorizations are unique and both triangles in the diagram above commute.
\end{proof}

\begin{lemma}\mylabel{lemma-factorization-id}
   Let $h:GX\to GY$ be a homomorphism with $X$ countable. If $C\subseteq Y$ is a countable subgraph such that $h(\id_X)\in GC\subseteq GY$ then $h(GX)\subseteq GC$.
\end{lemma}

\begin{proof}
  Suppose, to the contrary, that there exists a countable graph $D$, such that $C\subsetneq D\subseteq Y$, and an element $y\in GD\setminus GC$ which is in the image of $h$. Let $D_*$ be the full graph with the same vertices as $D$ and all possible edges. We may put these into the following diagram.
  $$
    \xymatrix{
      {\{\id_X\}} \ar[r]
          \ar@{}[d]|{\rotatebox{-90}{$\subseteq$}} &
        GC \ar@{}[r]|{\rotatebox{0}{$\subseteq$}} &
        GD \ar@{}[r]|{\rotatebox{0}{$\subseteq$}}
          \ar[dr]|{\rotatebox{-45}{$\subseteq$}}   &  %(0.25)*+{ \ \ } &
        GY  \ar[d] \\
      GX \ar[rrr]_{h'} \ar[rrru]_h \ar@{-->}[ur]^{\tilde{h'}} &&&
        GD_*
    }
  $$
  The top line, the left vertical inclusion and the $h$ are given by assumptions. The right vertical map is induced by some extension of the inclusion $D\subseteq D_*$ to $Y$. The homomorphism $h'$ is the composition of $h$ and this extension.

  Lemma \ref{lemma-factors}, applied to $h'$, $X$ as $W$ and the inclusion $C\subseteq D_*$ as the monomorphism $\varphi$, gives us the dashed homomorphism $\tilde{h'}$. The central trapezoid commutes. Thus the image of $h'$ in $GD_*$ is contained in $GC$. This is a contradiction since if $y\in GD$ is not in $GC$ then it is also not in $GC$ when viewed as an element of $GD_*$.
\end{proof}

We will need the following immediate consequence of this Lemma.

\begin{corollary}\mylabel{corollary-factorization}
  If $h:GX\to GY$ is a homomorphism with $X$ countable then there exists a countable subgraph $C\subseteq Y$ such that $h$ factors through $GC\subseteq GY$.
\end{corollary}

Functoriality of $G$ gives us a natural homomorphism $$\gamma:\mathbb{Z}[\Hom_{\g}(X,Y)]\to\Hom(GX,GY).$$

\rem \mylabel{remark-colimits}
  Lemma \ref{lemma-left-multiplication} implies
  that $\gamma$ is an isomorphism when both $X$ and $Y$ are countable. Corollary \ref{corollary-factorization} implies that it is enough that $X$ is countable since then
$$
  \mathbb{Z}[\Hom(X,Y)]\cong\mathbb{Z}[\Hom(X,\colim_{C\in[Y]}C)]\cong
  \colim_{C\in [Y]}\mathbb{Z}[\Hom(X,C)]
  \overset{\colim\gamma}{\underset{\cong}{-\!\!\!-\!\!\!-\!\!\!\longrightarrow}}
$$

$$
  \to\colim_{C\in[Y]}\Hom(GX,GC)
  \overset{\cong}{\longrightarrow}\Hom(GX, \colim_{C\in[Y]}GC)\cong
  \Hom(GX,GY).
$$
The last arrow being an isomorphism is equivalent to Corollary \ref{corollary-factorization}.

\begin{lemma}
\mylabel{lemma-limits-zs}
  Let $\{S_i\}_{i\in I}$ be a diagram of sets. Let
  $\lambda:\mathbb{Z}[\lim S_i]\to\lim\mathbb{Z}[S_i]$ be defined by the universal property of limits. If $I$ is codirected then $\lambda$ is one-to-one and if $I$ is countably codirected then $\lambda$ is an isomorphism.
\end{lemma}
\begin{proof}
  Let $a=\sum k_ss\in\mathbb{Z}[S]$ be a non-zero element of a free group $\mathbb{Z}[S]$ with some basis $S$. Let $|a|$ denote the {\em support} of $a$, consisting of those $s$ for which $k_s\neq 0$. If $a\in\mathbb{Z}[\lim S_i]$ then, since $I$ is codirected, there exists an $i\in I$ such that $|a|$ projects injectively into $S_i$ and therefore $\lambda(a)\neq 0$.

  Let $a=(a_i)\in\lim\mathbb{Z}[S_i]$. If there exists a sequence $i_n$, $n\in\mathbb{N}$, such that the supremum of the cardinalities of $|a_{i_n}|$ is $\omega_0$ then, since $I$ is countably codirected, there exists $S_{i_0}$ which maps to all $S_{i_n}$; but then $|a_{i_0}|$ must be infinite, a contradiction. Let $i_m\in I$ be such that $|a_{i_m}|$ is largest possible. Then for each $i<i_m$ the map $S_i\to S_{i_m}$ restricts to a bijection $|a_i|\to|a_{i_m}|$, hence the inclusions $|a_i|\subseteq S_i$ lift in a coherent way to
  $\lim_{i<i_m} S_i$. Since the set $\{i\mid i<i_m\}$ is coinitial in $I$, we have $\lim_{i<i_m}S_i=\lim S_i$, and therefore $a$ is in the image of $\lambda$.
\end{proof}

\begin{thm} \mylabel{theorem-embedding-graphs}
  There exists a functor $G$ from the category of graphs to the category of groups which induces natural isomorphisms
  $$
  \gamma:\mathbb{Z}[\Hom_{\g}(X,Y)]\overset{\cong}{\longrightarrow}\Hom(GX,GY)
  $$
\end{thm}

\begin{proof}
  We have a chain of isomorphisms
  $$
  \mathbb{Z}[\Hom(X,Y)]\cong\mathbb{Z}[\Hom(\colim_{C\in[X]}C,Y)]\cong
  \mathbb{Z}[\lim_{C\in[X]}\Hom(C,Y)]{\overset{\lambda}{\longrightarrow}}
  $$
  $$
  \to\lim_{C\in[X]}\mathbb{Z}[\Hom(C,Y)]{\overset{\lim\gamma}{-\!\!\!\longrightarrow}}
  \lim_{C\in[X]}\Hom(GC,GY)\cong
  $$
  $$
  \cong\Hom(\colim_{C\in[X]}GC,GY)\cong\Hom(GX,GY).
  $$
  Lemma \ref{lemma-limits-zs} implies that $\lambda$ is an isomorphism, Remark \ref{remark-colimits} implies that $\lim\gamma$ is an isomorphism as the limit of isomorphisms, and the last isomorphism follows from the definition of $GX$.
\end{proof}

\section{A finite approximation}
\mylabel{section-approximation}

The functor $G$, constructed in the preceding section, would be more convenient if for countable $X$ the group $GX$ were also countable. In general this might be difficult. However, in some cases it requires only a slight modification of the construction of $G$. Let $\Gamma_{fin}$ be the full subcategory of $\Gamma$ consisting of finite graphs. Let $A_{fin}\subseteq A$ be the corresponding countable subring and let $\mathbb{A}_{fin}$ be the countable group described in Theorem \ref{theorem-corner}. We may assume that $\mathbb{A}_{fin}$ is constructed in the same process as $\mathbb{A}$, so that $\mathbb{A}_{fin}\subseteq\mathbb{A}$ as $A_{fin}$-modules. We construct $G_{fin}X$ as in Definition \ref{definition-g} replacing $\mathbb{A}$ with $\mathbb{A}_{fin}$, $\Gamma$ with $\Gamma_{fin}$ and $[X]$ with the subcategory $[X]_{fin}$ consisting of finite subgraphs.

As in the preceding section we have a natural transformation
$$
  \gamma:\mathbb{Z}[\Hom_{\g}(X,Y)]\to\Hom(G_{fin}X,G_{fin}Y).
$$

\begin{theorem}\mylabel{theorem-embedding-graphs-fin}
  Let $G_{fin}$ and $\gamma$ be as above.
  \begin{enumerate}
    \item \label{tegf-cardinality}If $X$ is infinite then $|X|=|G_{fin}X|$.
    \item \label{tegf-transformation} The inclusion of $A_{fin}$-modules $\mathbb{A}_{fin}\subseteq\mathbb{A}$ yields a natural transformation $h:G_{fin}\to G$ consisting of inclusions.
    \item \label{tegf-11} $\gamma$ is one-to-one.
    \item \label{tegf-iso-source} If $X$ is finite then $\gamma$ is an isomorphism.
    \item \label{tegf-iso-hom} If $\Hom(X,Y)$ is finite then $\gamma$ is an isomorphism.
  \end{enumerate}
\end{theorem}

\begin{proof}
  In the colimit construction of $G_{fin}X$ we have a countable $G_{fin}C\subseteq\aa_{fin}$ for every finite subgraph $C$ of $X$, hence (\ref{tegf-cardinality}). For $C$ in $[X]_{fin}$ we have
  $\id_C\in A_{fin}$, hence $\id_C\aa_{fin}\subseteq \id_C\aa$, that is, $G_{fin}C\subseteq GC$. Then $h:G_{fin}X\to GX$ is the colimit map. The poset $[X]_{fin}$ is not cofinal in $[X]$ but the injectivity of $h$ is retained since Lemma \ref{lemma-injective} implies that $GX$ is a directed colimit of inclusions. This proves (\ref{tegf-transformation}). The only reason $[X]$ had to contain the countable subgraphs of $X$, in the preceding section, is the countable codirectedness of the diagram $I$ in Lemma \ref{lemma-limits-zs}. With only finite subgraphs we have in $[X]_{fin}$ we obtain the injectivity part of that lemma, hence (\ref{tegf-11}) is implied by the finite version of Theorem \ref{theorem-embedding-graphs}. Also (\ref{tegf-iso-source}) is a finite version of Remark \ref{remark-colimits} which makes no use of infinite subgraphs.

  It remains to prove the surjectivity part of (\ref{tegf-iso-hom}). An element $a$ of $$\Hom(G_{fin}X,G_{fin}Y)\cong\lim_{C\in[X]_{fin}}\Hom(G_{fin}C,G_{fin}Y)$$ is a sequence $a=(a_C)_{C\in[X]_{fin}}$ of compatible homomorphisms. We know by (\ref{tegf-iso-source}) that each $a_C$ may be uniquely represented as
  $a_C=\sum_{\sigma\in I_C\setminus\{0\}}k_\sigma\sigma$ where $I_C$ is a finite set of maps $\sigma:C\to Y$ and $0$, and the $k_\sigma$ are nonzero integers. An inclusion $B\subseteq C$ induces restriction of $a_C$ to $a_B$ which further induces an epimorphism $\nu^C_B:I_C\to I_B$, defined as follows. Fix $\sigma_0\in I_C\setminus\{0\}$ and let $\tau_0=\sigma_0|_B$. If the sum of $k_\sigma$ over those $\sigma\in I_C\setminus\{0\}$ for which $\sigma|_B=\tau_0$ is $0$ then $\nu^C_B(\sigma_0)=0$, otherwise $\nu^C_B(\sigma_0)=\tau_0$. Define $\nu^C_B(0)=0$. Let $I=\lim_{C\in[X]_{fin}}I_C$. If $\varphi\neq 0$ is an element of $I$ then it corresponds to a map $\varphi:X\to Y$. If $\Hom(X,Y)$ is finite then $I$ is finite, hence $a$ can be written as
  $\sum_{\varphi\in I\setminus\{0\}}k_\varphi\varphi$, which proves (\ref{tegf-iso-hom}).
\end{proof}

\section {Some quick applications}
\mylabel{section-applications}

In this section we present a collection of applications of the functors $G$ and $G_{fin}$, constructed in the preceding sections.

\intertitle{Rigid systems of groups}

Vop\v enka, Pultr and Hedrl\'in proved in \cite{vopenka} that for every infinite cardinal $\kappa$ there exists a graph $X$ of cardinality $\kappa$ such that $\Hom(X,X)=\{\id_X\}$. An easy modification yields a rigid system of graphs $\{X_i\}_{i<2^\kappa}$. Rigid means that if $f:X_i\to X_j$ is a map then $i=j$ and $f$ is the identity. Applying Theorem \ref{theorem-embedding-graphs-fin} we see that $\{G_{fin}X_i\}_{i<2^\kappa}$ is a rigid system of groups, meaning that if $h:GX_i\to GX_j$ is a nonzero homomorphism then $i=j$ and $h(x)=rx$ for some integer $r$. Each group of this system has cardinality $\kappa$. This is an alternative proof of the result of Shelah \cite[Theorem 2.1]{shelah}.

\intertitle{A class of groups}

For every infinite cardinal $\kappa$ we construct a graph $X_\kappa$ whose vertices are ordinals $\alpha<\kappa$ and edges $\alpha\to\beta$ are relations $\alpha<\beta$. Applying Theorem \ref{theorem-embedding-graphs-fin} we obtain a proper class of groups $G_{fin}X_\kappa$ such that for $\kappa<\lambda$ we have $\Hom(G_{fin}X_\lambda,G_{fin}X_\kappa)=0$. Note that using \cite{vopenka} we may modify the construction of $X_\kappa$ so that the additional condition $\Hom(G_{fin}X_\kappa,G_{fin}X_\kappa)=\mathbb{Z}$ is satisfied. We learned from G\"obel that an analogous class may be constructed inductively using Theorem 1 of \cite{gobel}.

\intertitle{Generalized pure subgroups}

If $\kappa$ is an infinite cardinal, a subgroup $N$ of $M$ is said to be $\kappa$-{\em pure} if $N$ is a direct summand of every subgroup $N'$ such that $N\subseteq N'\subseteq M$ and $|N'/N|<\kappa$. Megibben proved in \cite[Proposition 3.1]{megibben} that for every infinite cardinal $\kappa$ there exists a group containing a $\kappa$-pure subgroup which is not $\kappa^+$-pure. Here, $\kappa^+$ denotes the successor cardinal of $\kappa$.

Let $\lambda>\kappa$ be an ordinal and let $X_\alpha$ be graphs as in the preceding subsection constructed for ordinals $\alpha<\lambda$. Let $W_\lambda$ be the wedge sum of $X_\alpha$, $\alpha<\lambda$, with the $0$'s identified. We claim that $G_{fin}W_\kappa\subseteq G_{fin}W_\lambda$ is $\kappa$-pure but not $\kappa^+$-pure. If $N'$ contains $G_{fin}W_\kappa$ and $|N'/G_{fin}W_\kappa|<\kappa$ then $N'$ is generated by $G_{fin}W_\kappa$ and less than $\kappa$ elements, hence is contained in $G_{fin}Y$ for $Y=W_\kappa\cup Y_0$ where $Y_0$ is some subgraph of $W_\lambda$ of cardinality less than $\kappa$. Clearly $W_\kappa$ is a retract of $Y$ and therefore $G_{fin}W_\kappa$ is a direct summand of $G_{fin}Y$. On the other hand $G_{fin}W_{\kappa+1}$ has cardinality $\kappa$ and we have no maps from $W_{\kappa+1}$ to $W_{\kappa}$, hence $\Hom(G_{fin}W_{\kappa+1},G_{fin}W_{\kappa})=0$ and therefore $G_{fin}W_\kappa$ is not a direct summand of $G_{fin}W_{\kappa+1}$.

Megibben's examples are $p$-groups. Here we obtain torsion free groups.

\intertitle{Chains of group localizations}

In \cite{przezdziecki-chains} it is shown that a self-free abelian group, constructed by Dugas in the proof of \cite[Theorem 2.1]{dugas}, yields a chain of groups $M_\alpha$, $\alpha<\lambda$, where $\lambda$ is any nonmeasurable cardinal. Inclusions in this chain, $M_\alpha\subseteq M_\beta$, for $\alpha<\beta<\lambda$, are localizations in the sense that they induce isomorphisms $\Hom(M_\beta,M_\beta)\cong\Hom(M_\alpha,M_\beta)$.

Any ordinal $\lambda$, viewed as a category, fully embeds into $\g$, as any small category does by \cite[8.5 and 8.6 on page 53 and Theorem on page 104]{trnkova-book}. Thus Theorem \ref{theorem-embedding-graphs} removes the above restriction to nonmeasurable cardinals. The price we pay here is that the groups which appear in the chain are not self-free as was the case in \cite{przezdziecki-chains}. Assuming the negation of Vop\v enka's principle (such an assumption is consistent with ZFC) \cite[Lemma 6.3]{adamek-rosicky} implies that the ordered class of all ordinals, considered as a category, fully embeds into $\g$ and then Theorem \ref{theorem-embedding-graphs} yields an unbounded chain indexed by all ordinals. In fact Theorem \ref{theorem-embedding-graphs} and \cite[Lemma 6.3]{adamek-rosicky} imply that nonexistence of such a chain in $\a$ is equivalent to Vop\v enka's principle.

\section{Orthogonal subcategory problem in the category of abelian groups}
\mylabel{section-orthogonal}

In this section we prove that analogues of two large cardinal principles, formulated in the category of abelian groups, are equivalent to their original statements in the category of graphs: Every orthogonality class in $\a$ is reflective if and only if weak Vop\v enka's principle holds (Proposition \ref{proposition-weak-vopenka}); and every orthogonality class in $\a$ is a small-orthogonality class if and only if Vop\v enka's principle holds (Proposition \ref{proposition-vopenka}).

An object $X$ and a morphism $f:A\to B$ in a category $\mathcal{C}$ are {\em orthogonal}, denoted $f\perp X$, if $f$ induces a bijection
\begin{equation}\label{equation-orthogonal-groups}
  \Hom_\mathcal{C}(B,X)\overset{\cong}{\longrightarrow}\Hom_\mathcal{C}(A,X).
\end{equation}
For a class $\mathcal{D}$ of objects, $\mathcal{D}^\perp$ denotes the class of morphisms orthogonal to every member of $\mathcal{D}$, and analogously $\mathcal{E}^\perp$ denotes the class of objects orthogonal to every element of the class $\mathcal{E}$ of morphisms. Classes of objects of the form $\mathcal{D}=\mathcal{E}^\perp$ are called {\em orthogonality classes}, and when $\mathcal{E}$ is a set, not a proper class, $\mathcal{D}=\mathcal{E}^\perp$ is a called a {\em small-orthogonality class}. A pair $(\mathcal{E},\mathcal{D})$ of a class $\mathcal{E}$ of morphisms and a class $\mathcal{D}$ of objects such that $\mathcal{D}=\mathcal{E}^\perp$ and $\mathcal{E}=\mathcal{D}^\perp$ is called an {\em orthogonal pair}. It is easy to see that $\mathcal{D}^{\perp\perp\perp}=\mathcal{D}^\perp$ and $\mathcal{E}^{\perp\perp\perp}=\mathcal{E}^\perp$.

A class $\mathcal{D}$ of objects is called {\em reflective} if for every object $A$ in the category $\mathcal{C}$ there exists a morphism $f:A\to\bar{A}$ in $\mathcal{D}^\perp$ such that $\bar{A}$ is in $\mathcal{D}$. The morphism $f$ is called a {\em reflection} of $A$ in $\mathcal{D}$. The reflections are unique up to isomorphism. The following is an observation dating back to \cite[Exercise 3.F]{freyd-abelian-categories}.

\rem \label{remark-reflector} If every object of $\mathcal{C}$ has a reflection in $\mathcal{D}$ we obtain a functor $L:\mathcal{C}\to\mathcal{C}$ which assigns to each object $X\in\mathcal{C}$ a reflection of $X$ in $\mathcal{D}$. Such a functor $L$ is called a {\em reflector} in \cite{freyd-abelian-categories} or, in research inspired by algebraic topology (see e.g. \cite{casacuberta-advances}), a {\em localization}. The localization is a left adjoint of the inclusion of $\mathcal{D}$ into $\mathcal{C}$ as a full subcategory. It comes with a natural transformation $\eta:\rm{Id}\to L$ such that $\eta_{LX}=L\eta_X:LX\to LLX$ is an isomorphism for every $X$ in $\mathcal{C}$.

Every localization is associated with an orthogonal pair $(\mathcal{E},\mathcal{D})$ such that $\mathcal{E}$ consists of those $f$ for which $Lf$ is an isomorphism and $\mathcal{D}$ consists of those $X$ for which $\eta_X:X\to LX$ is an isomorphism. This orthogonal pair uniquely determines the localization. The converse question whether an orthogonal pair $(\mathcal{E},\mathcal{D})$ is associated with a localization, that is whether $\mathcal{D}$ is reflective, is known as the {\em orthogonal subcategory problem}. General answers to this problem turned out to fit into the hierarchy of large cardinals \cite[page 472]{kanamori}. They are formulated in the category of graphs \cite[Chapter 6]{adamek-rosicky} as follows:

\noindent{\bf Weak Vop\v enka's principle:} Every orthogonality class in the category of graphs is reflective.

\noindent{\bf Vop\v enka's principle:} Every orthogonality class in the category of graphs is a small-orthogonality class (in particular reflective).

Both principles imply analogous statements in every locally presentable category, in particular in $\a$: see \cite[Theorem 6.22 and Corollary 6.24]{adamek-rosicky}.

\begin{lemma}\mylabel{lemma-epi}
  Let $\mathcal{D}_0$ be a class of objects in $\g$ or $\a$, and $\mathcal{E}=\mathcal{D}_0^\perp$. For every object $A$ there exists a diagram
  $$A\overset{\alpha}{\longrightarrow}A'\overset{e}{\longrightarrow}Z_A$$
  where $\alpha$ is an epimorphism in $\mathcal{E}$, $e$ is a monomorphism, and $Z_A$ is a product of elements of $\mathcal{D}_0$.
\end{lemma}
\begin{proof}
  For every pair $x$, $y$ of elements of $A$ we choose, if it exists, a morphism $\alpha_{xy}:A\to Z_{xy}$ such that $Z_{xy}$ is in $\mathcal{D}_0$ and $\alpha_{xy}(x)\neq \alpha_{xy}(y)$. Let $Z_A=\prod Z_{xy}$ and let
  $d:A\to Z_A$ be the diagonal of $\alpha_{xy}$'s. Define $A'=d(A)$, let $\alpha:A\to A'$ be the restriction of $d$ and let $e$ be the inclusion $A'$ into $Z_A$. For every $Z$ in $\mathcal{D}_0$ the morphism $\alpha$ induces a bijection
  $$\alpha^*:\Hom(A',Z)\to\Hom(A,Z),$$
  hence $\alpha$ is in $\mathcal{E}$.
\end{proof}

Let $(\mathcal{E},\mathcal{D})$ be an orthogonal pair in $\g$. Theorem \ref{theorem-embedding-graphs} implies that every $f$ in $G\mathcal{E}$ is orthogonal to every $M$ in $G\mathcal{D}$, hence the pair $(G\mathcal{E},G\mathcal{D})$ extends to an orthogonal pair $(\overline{\mathcal{E}},\overline{\mathcal{D}})$ defined by $\overline{\mathcal{E}}=G\mathcal{D}^\perp$ and $\overline{\mathcal{D}}=\overline{\mathcal{E}}^\perp$.

\begin{lemma}
  If the orthogonal pair $(\overline{\mathcal{E}},\overline{\mathcal{D}})$, defined above, is associated with a localization $L$ then $(\mathcal{E},\mathcal{D})$ is also associated with a localization.
\end{lemma}
\begin{proof}
  Remark \ref{remark-reflector} implies that it is enough to find for every graph $X$ a map
  $X\to Y$ in $\mathcal{E}$ such that $Y$ is in $\mathcal{D}$. Applying Lemma \ref{lemma-epi} to $\mathcal{D}_0=\mathcal{D}$ and $A=X$ we obtain a map $X\to X'$ in $\mathcal{E}$ such that $X'$ embeds into some element of $\mathcal{D}$. Now it is enough to find a map $X'\to Y$ in $\mathcal{E}$ such that $Y$ is in $\mathcal{D}$. Therefore, in the remainder of the proof, we may assume that $X$ embeds into an element of $\mathcal{D}$.

  Applying Lemma \ref{lemma-epi} to $\mathcal{D}_0=G\mathcal{D}$ and $A=LGX$ we obtain an epimorphism $a:LGX\to A'$ in $\overline{\mathcal{E}}$ such that $A'$ embeds into a product of elements of $G\mathcal{D}$. Since $LGX$ is in $\overline{\mathcal{D}}$ and $\alpha$ is an epimorphism in $\overline{\mathcal{E}}$ we see that $\alpha$ is an isomorphism and therefore we obtain an embedding
  $$e:LGX\to\prod GZ_i$$

  Theorem \ref{theorem-embedding-graphs} implies that each composition $h_i$ as in
  \begin{equation}
  \mylabel{equation-product}
    \raisebox{0pt}[30pt][0pt]{}
    \xymatrix{
      GX \ar[r]^\eta \ar@/^2pc/[rrr]^{h_i} &
      LGX \ar[r]^e &
      {\prod GZ_i} \ar[r]^{\pi_i} &
      GZ_i
    }
  \end{equation}
  is a combination $\sum_{\sigma\in I_i}k_\sigma\sigma$ of homomorphisms induced by maps $\sigma_i:X\to Z_i$. We replace each $GZ_i$ in the product with $\prod_{\sigma\in I_i}GZ_i$ and each $h_i$ with the diagonal $GX\to\prod_{\sigma\in I_i}GZ_i$ of the maps $\sigma\in I_i$. Since $\eta$ is in $\overline{\mathcal{E}}$ and the products are in $\overline{\mathcal{D}}$, each diagonal factors through $\eta$ and therefore we may assume that each $h_i$ in (\ref{equation-product}) is induced by a map of graphs $\sigma_i:X\to Z_i$.

  We look at the following diagram:
  \begin{equation}
  \mylabel{equation-xlx}
  \xymatrix{
    &
      {\mbox{\ \ \ \ \ \ \ \ }G\prod Z_i = GZ} \ar[dr]^\pi \\
    GX \ar[r]^\eta \ar[ur]^{G\varphi} &
      LGX \ar@{-->}[u]^{e'} \ar[r]^{e} &
      {\prod GZ_i}
  }
  \end{equation}
  The graph $Z$ is defined as $\prod Z_i$, and the map $\varphi$ is the diagonal of $\sigma_i:X\to Z_i$ as above. At the beginning of the proof we assumed that $X$ embeds into an element of $\mathcal{D}$, hence we may assume that at least one of the $\sigma_i$'s is injective and therefore $\varphi$ is one-to-one. The homomorphism $\pi$ is the product homomorphism. The homomorphism $e'$ exists since $GZ$ is in $\overline{\mathcal{D}}$ and therefore $G\varphi$ uniquely factors through $\eta$ which is in $\overline{\mathcal{E}}$. Since $e\eta=\pi e'\eta$ and $\prod GZ_i$ is in $\overline{\mathcal{D}}$ and $\eta$ in $\overline{\mathcal{E}}$, we have $e=\pi e'$. We see that $e'$ is a monomorphism since $e$ is.

  Up to this point we have reduced our situation to the case when there exists a graph $Z$ in $\mathcal{D}$ and a monomorphism $\varphi:X\to Z$ which induces a monomorphism $e':LGX\to GZ$ as in (\ref{equation-xlx}). In the remainder we enlarge $X$ to an $X^\bullet$ such that $\varphi(X)\subseteq X^\bullet\subseteq Z$, the inclusion $\varphi(X)\subseteq X^\bullet$ is in $\mathcal{E}$ and $X^\bullet$ is in $\mathcal{D}$.

  We identify $X$ with $\varphi(X)\subseteq Z$. Since the class $\mathcal{E}$ is closed under colimits we see that the set of subgraphs $X'$ such that $X\subseteq X'\subseteq Z$ and the inclusion $X\subseteq X'$ is in $\mathcal{E}$ has a maximal element $X^\bullet$. We have $LGX^\bullet\cong LGX$ and the properties of Diagram (\ref{equation-xlx}) are retained when we replace $X$ with $X^\bullet$; the graphs $Z_i$ and $Z$ remain the same.

  The proof will be complete once we demonstrate that $X^\bullet$ is in $\mathcal{D}$. It is enough to show that for any $\varepsilon:W_1\to W_2$ in $\mathcal{E}$ and any map $f:W_1\to X^\bullet$ the dashed map below exists and is unique.
  \begin{equation}
  \mylabel{equation-two-squares}
  \xymatrix{
    W_1 \ar[r]^f \ar[d]_\varepsilon &
      X^\bullet \ar[d]^{\varepsilon'} \ar@{}[rr]|*+{\subseteq} &
      &
      Z \\
    W_2 \ar[r] \ar@{-->}[ur] &
      P \ar@{..>}[urr]^\mu \ar[r]_{\alpha} &
      P' \ar[r]_(0.38)e \ar@{..>}[ur]_{\mu'} &
      Z\times Z_P \ar[u]_\pi
  }
  \end{equation}
  The uniqueness is clear from the uniqueness of the composition with the inclusion $X^\bullet\subseteq Z$. Let $P$ be the pushout of $\varepsilon$ and $f$. Then $\varepsilon'$ is in $\mathcal{E}$ and therefore, since $Z$ is in $\mathcal{D}$, the map $\mu$ exists. Lemma \ref{lemma-epi} implies the existence of the epimorphism $\alpha$ in $\mathcal{E}$ and of the monomorphism $e$. The map $\mu'$ exists since $\alpha$ is in $\mathcal{E}$ and $Z$ is in $\mathcal{D}$. Commutativity of the triangles involving $\mu$ or $\mu'$ follows from their uniqueness under $X^\bullet$.

  We apply $G$ to the right part of (\ref{equation-two-squares}) and factor some arrows through $LGX^\bullet$ to obtain the diagram
  $$
  \xymatrix{
    GX^\bullet \ar[d]_{G\alpha\varepsilon'} \ar@{}[r]|*+{\subseteq} &
      LGX^\bullet \ar[dr]_g \ar[r]^{e'}_{\subseteq} &
      GZ \\
    GP' \ar[rr]^{Ge} \ar[ur]_h &
    &
    G(Z\times Z_P) \ar[u]_{G\pi}
  }
  $$
  We proved that $\alpha\varepsilon'$ is in $\mathcal{E}$, hence $G\alpha\varepsilon'$ is in $\overline{\mathcal{E}}$, and since also $LGX^\bullet$ is in $\overline{\mathcal{D}}$, we know that $h$ exists. Since $GX^\bullet\subseteq LGX^\bullet$ is in $\overline{\mathcal{E}}$ and $G(Z\times Z_P)$ is in $\overline{\mathcal{D}}$, we obtain $g$.

  Since $Ge$ is one-to-one we see that $h$ has to be a monomorphism and therefore $\mu'$ in Diagram (\ref{equation-two-squares}) is also a monomorphism. Maximality of $X^\bullet$ implies that $\alpha\varepsilon'$ is an isomorphism, hence the dashed arrow in (\ref{equation-two-squares}) exists, which completes the proof.
\end{proof}

As a corollary we obtain the following.

\begin{proposition}
\mylabel{proposition-weak-vopenka}
  Assuming the negation of weak Vop\v enka's principle, there exists a nonreflexive orthogonality class in the category of abelian groups.
\end{proposition}

\begin{proposition}\mylabel{proposition-vopenka}
  Assuming the negation of Vop\v enka's principle there exists an orthogonality class in the category of abelian groups which is not a small-orthogonality class.
\end{proposition}
\begin{proof}
  Since the functor $G$ preserves countably directed colimits the proof of \cite[Proposition 8.8]{przezdziecki-groups} applies.
\end{proof}

The converses of Propositions \ref{proposition-weak-vopenka} and
\ref{proposition-vopenka} follow from
\cite[Theorem 6.22 and Corollary 6.24(iii)]{adamek-rosicky}.

\section{Non-reflective subcategories of the stable homotopy category}
\mylabel{section-stable}

We prove, under negation of weak Vop\v enka's Principle, that, in the stable homotopy theory of spectra $\mathcal{S}$, there exists a semiorthogonality class which is not associated with a localization. This is closely related to the Hovey-Palmieri-Strickland problem \cite[Section 3.2]{hovey-axiomatic}.

Let $X$ be a spectrum and $f:A\to B$ be a map between spectra. The notion of orthogonality (\ref{equation-orthogonal-groups}) is not suitable here. Instead, we should look at the map of function spectra, induced by $f$:
\begin{equation}
  f^*:F(B,X)\longrightarrow F(A,X)
\end{equation}
We have two meaningful notions of orthogonality here. We say that $f$ is {\em semiorthogonal} to $X$ and write $f\curlywedge X$ if $f^*$ induces isomorphisms of the stable homotopy groups in the nonnegative degrees. We say that $f$ is {\em orthogonal} to $X$ and write $f\perp X$ if $f^*$ is a weak equivalence -- induces isomorphisms of all the homotopy groups.

These two notions of (semi)orthogonality lead to {\em (semi)orthogonal pairs}: Recall that $(\mathcal{E},\mathcal{D})$ is a (semi)orthogonal pair if $\mathcal{E}$ consists of all the morphisms (semi)orthogonal to every object in $\mathcal{D}$ and conversely. As in the preceding section if $\mathcal{C}$ is a class of objects (morphisms) then $\mathcal{C}^\curlywedge$ denotes the class of all morphisms (objects) semiorthogonal to every member of $\mathcal{C}$. A functor $L:\mathcal{S}\to\mathcal{S}$ is called a localization associated with a (semi)orthogonal pair $(\mathcal{E},\mathcal{D})$ if $\mathcal{E}$ is the class of $L$-equivalences (i.e. those morphisms $f$ for which $Lf$ is a weak equivalence) and $\mathcal{D}$ is the class of objects weakly equivalent to $LX$ for some $X$ (the local objects).

Let us note that this concept of semiorthogonality is analogous, but not equivalent, to the one found in \cite[Section 1.3]{casacuberta-hovey}.

One should note that most of the contemporary literature considers only those localizations which are associated with the orthogonal pairs. Those localizations which are associated with semiorthogonal pairs only do not preserve exact sequences nor commute with the suspension, main examples are the stable analogues of the Postnikov sections. Such localizations are allowed for example in \cite[Paragraph 2.2]{bousfield}, \cite[Diagram 6.4]{casacuberta-connective}, \cite[Remark 4.4]{dwyer} and others.

Hovey, Palmieri and Strickland asked, in a different language, if every orthogonal pair, in a stable homotopy category, is associated with a localization
\cite[Section 3.2]{hovey-axiomatic}. The question was answered positively, under Vop\v enka's Prinsiple, by Casacuberta, Guti\'errez and Rosick\'y \cite{casacuberta-hovey} -- they constructed localizations associated to orthogonal as well as semiorthogonal pairs. Their result is valid in stable model categories which admit combinatorial models. Theorem \ref{theorem-stable} gives, under negation of weak Vop\v enka's Principle, a negative answer for semiorthogonal pairs in the stable homotopy theory of spectra $\mathcal{S}$. This result may suggest that the answer to the Hovey-Palmieri-Strickland question depends on set theory. On the other hand one should note that the Hovey-Palmieri-Strickland problem has a positive answer in relatively less complicated stable homotopy categories -- in derived categories $D(R)$ of Noetherian rings $R$: it follows from a result of Neeman who proved \cite{neeman} that if $R$ is Noetherian then the colocalizing subcategories in $D(R)$ form a set.

\begin{theorem}\label{theorem-stable}
  Assuming negation of weak Vop\v enka's Principle there exists a semiorthogonal pair in $\mathcal{S}$ which is not associated with a localization.
\end{theorem}

\begin{proof}
  Proposition \ref{proposition-weak-vopenka} states that the negation of weak Vop\v enka's Principle implies existence of an orthogonal pair
  $(\mathcal{E},\mathcal{D})$ in the category of abelian groups $\a$ which is not associated with any localization.
  We have a full embedding $H : \a\to\mathcal{S}$ which sends a group $A$ to the Eilenberg-Mac Lane spectrum $HA$ where $(HA)_n = K(A, n)$. Some stable homotopy groups of the spectrum of functions between the Eilenberg-Mac
  Lane spectra are described below:
$$
\pi_n(F(HA,HB)) = [\Sigma^nHA,HB] =
  \left\{\begin{array}{ll}
    0 & n > 0 \\
    \Hom(A,B) & n = 0 \\
    \Ext(A,B) & n = -1 \\
  \end{array}\right.
$$
This implies that the embedding H takes a pair $f$ and $A$ such that $f\perp A$ in $\a$ to a pair $Hf$ and $HA$ such that $Hf\curlywedge HA$ in $\mathcal{S}$.
In other words $H$ preserves the semiorthogonality (not the orthogonality). This implies
that $(H\mathcal{E},H\mathcal{D})$ extends to a semiorthogonal pair $(H\mathcal{D}^{\curlywedge},H\mathcal{D}^{\curlywedge\curlywedge})$ in $\mathcal{S}$.

Suppose that the pair $(H\mathcal{D}^{\curlywedge},H\mathcal{D}^{\curlywedge\curlywedge})$ corresponds to a localization $L$. Then for any group $A$ we
have a localization map
$$ HA \to LHA$$
In the remainder of the proof we obtain a contradiction by demonstrating that the corresponding homomorphisms of abelian groups
\begin{equation}\label{equation-localization-groups}
 A = \pi_0HA \to \pi_0LHA
\end{equation}
define a localization associated with the orthogonal pair $(\mathcal{E},\mathcal{D})$.
Let $S$ be the sphere spectrum. Since $\Sigma S \to 0$ belongs to $H\mathcal{D}^\curlywedge$ we see that the homotopy
groups of $LHA$ are trivial above $0$. Let $Z \to LHA$ be the $(-1)$-connected cover. Then
$Z$ has its homotopy concentrated in dimension $0$ and therefore $Z = HB$ for some group $B$. Let $F$ be the homotopy fiber of the cover. We obtain a fibration
$$F \to HB \to LHA$$
and for any group $C$ we have the following exact sequence
$$[HC,F] \to [HC,HB] \twoheadrightarrow [HC,LHA]$$
The right hand map is onto since $HC$ is its own $(-1)$-connected cover. Also $[HC, F] = 0$
since the homotopy of F is concentrated in the negative degrees. Therefore the right hand map is an isomorphism hence:
\begin{equation}\label{equation-stable-hom}
\Hom(C,B) \cong [HC,LHA]
\end{equation}
For any homomorphism of abelian groups $f : C_1 \to C_2$ in $\mathcal{E}$ we have
$$
  \xymatrix{
    \Hom(C_2,B) \ar[r]^{\cong}\ar[d]_{f^*} &
      [HC_2,LHA] \ar[d]^{\cong}_{Hf^*} \\
    \Hom(C_1,B) \ar[r]^{\cong} &
      [HC_1,LHA]
  }
$$
The horizontal arrows are the isomorphisms (\ref{equation-stable-hom}), $LHA$ is local and $Hf$ is in $H\mathcal{E}$ hence $Hf^*$ is an isomorphism and therefore $f^*$ is an isomorphism. Since $f$ is an arbitrary member of $\mathcal{E}$ we see that $B$ is in $\mathcal{D}$ and therefore $HB$ is in $H\mathcal{D}$.

Since $HB\to LHA$ is a $(-1)$-connected cover we obtain the following factorization of the localization map.
$$
HA \to HB \to LHA
$$
Since $HB$ is local it implies that $LHA$ is a retract of $HB$ which, together with the fact that $HB$ is the $(-1)$-connected cover of $LHA$, yields $HB \cong LHA$. Thus $L$ preserves the image of $H$. It implies that $L$ lifts to a localization in the category of abelian groups (\ref{equation-localization-groups}), associated with $(\mathcal{E},\mathcal{D})$, which does not exist by assumption. Therefore $L$ can not exist.
\end{proof}

\section{The isomorphism problem}
\mylabel{section-isomorphism}

It is natural to ask if an isomorphism $GX\cong GY$ implies $X\cong Y$. The affirmative answers to the analogous questions in the cases of full embeddings and almost full embeddings in the sense of \cite{trnkova-30} or \cite{przezdziecki-groups} are clear. We don't know the answer for the functor $G$ considered here. All we know is that $X$ has to be isomorphic to a retract of $Y$ and vice versa: if $h:GX\to GY$ is an isomorphism then $h=\sum_ik_i\sigma_i$ and $h^{-1}=\sum_jn_j\tau_j$, hence $\id_X=\sum_j\sum_in_jk_i\tau_j\sigma_i$, and therefore for certain $i$ and $j$ we have $\id_X=\sigma_i\tau_j$. In particular any counterexample to the isomorphism problem has to be infinite. A possible negative answer to this problem does not interfere with the results of Section \ref{section-orthogonal} since if $\mathcal{D}$ is an orthogonality class and $X$ belongs to $\mathcal{D}$ then all the retracts of $X$ belong to $\mathcal{D}$.

\end{document}